\documentclass[a4paper,12pt]{amsart}

\usepackage{amsmath,amssymb,amsthm,graphicx,pstricks,comment}
\usepackage[all]{xy}

\usepackage{lscape}
\usepackage{hyperref}
\usepackage{tikz}
\usetikzlibrary{arrows}
\usetikzlibrary{snakes}
\usetikzlibrary{shapes}

\author{Claire Amiot}
\address{Institut Fourier, 100 rue des maths, 38402 Saint Martin d'H\`eres}
\email{claire.amiot@ujf-grenoble.fr}\thanks{partially supported by the ANR project ANR-09-BLAN-0039-02.}

\numberwithin{equation}{section}

\newcommand{\Hom}{{\sf Hom }}\newcommand{\RHom}{\mathbf{R}{\sf Hom }}

\newcommand{\Ext}{{\sf Ext }}

\renewcommand{\mod}{{\sf mod \hspace{.02in}  }}
\newcommand{\fd}{{\sf fd \hspace{.02in}  }}
\newcommand{\Mod}{{\sf Mod \hspace{.02in} }}

\newcommand{\proj}{{\sf proj \hspace{.02in} }}

\newcommand{\gr}{{\sf gr \hspace{.02in} }}
\newcommand{\qgr}{{\sf qgr \hspace{.02in} }}
\newcommand{\grproj}{{\sf grproj \hspace{.02in} }}

\newcommand{\ten}{\otimes}
\newcommand{\lten}{\overset{\mathbf{L}}{\ten}}

\newcommand{\Talg}{{\sf T}}

\newcommand{\gldim}{{\sf gl.dim \hspace{.02in}}}

\newcommand{\Sing}{{\sf Sing}}

\newcommand{\Dd}{\mathcal{D}}

\newcommand{\Pp}{\mathcal{P}}

\newcommand{\SSS}{\mathbb{S}}

\newcommand{\Db}{\mathcal{D}^{\rm b}}

\newcommand{\bsm}{\begin{smallmatrix}}
\newcommand{\esm}{\end{smallmatrix}}

\newtheorem{theorem}{Theorem}[section]
\newtheorem*{thm*}{Th{\'e}or{\`e}me}

\newtheorem{corollary}[theorem]{Corollary}

\theoremstyle{remark}
\newtheorem{remark}[theorem]{Remark}

\theoremstyle{definition}
\newtheorem{definition}[theorem]{Definition}

\title{Preprojective algebras, singularity categories and orthogonal decompositions}

\begin{document}
\maketitle
\begin{abstract}
In this note we use results of \cite{Min11} and \cite{AIR11} to construct an embedding of the graded singularity category of certain graded Gorenstein algebras into the derived categories of coherent sheaves over its projective scheme. These graded algebras are constructed using the preprojective algebras of $d$-representation infinite algebras as defined in \cite{HIO12}. We relate this embedding to the construction of a semi-orthogonal decomposition of the derived category of coherent sheaves over the projective scheme of a Gorenstein algebra of parameter $1$ described in \cite{Orl05}.
\end{abstract}
\section{Introduction}

To a commutative graded Noetherian ring, Orlov associates in \cite{Orl05} the graded singularity category $\Sing^{\gr}(R)$ defined as the Verdier localization of the bounded derived category $\Db(\gr R)$ by the full triangulated subcategory of perfect complexes $\Db(\grproj R)$. This category is a graded analogue of the singularity category $\Sing (R)$ of $R$ which captures many properties of the singularities of the affine scheme ${\rm Spec}(R)$. Associated to the graded algebra $R$, another natural triangulated category to consider is the derived category $\Db(\qgr R)$ of graded tails of $R$, where $\qgr R$ is the quotient of the abelian category $\gr R$ of finitely generated graded $R$-modules by the subcategory $\fd\gr R$ of finite dimensional ones. By a classical theorem due to Serre, the category $\qgr R$ is equivalent to the category of coherent sheaves of the projective scheme ${\rm Proj}(R)$. When the algebra $R$ is Gorenstein, Orlov relates these two categories. More precisely, when the Gorentein parameter of $R$ is positive, there is an embedding \begin{equation}\label{Orlovintro} \xymatrix{\Sing^{\gr}(R)\ar@{^(->}[r] & \Db(\qgr R)}.\end{equation}
This applies in particular when $R$ is the ring of homogenous coordinates of an hypersurface in $\mathbb{P}^N$ which is a Fano variety.

More recently, in the context of non commutative algebraic geometry, Minamoto and Mori introduced the notion of quasi-Fano algebras. For such an algebra $\Lambda$, Minamoto contructs in \cite{Min11} a triangle equivalence between the derived category of the module category $\Mod \Lambda$ and the derived category of graded tails ${\sf QGr} \Pi$ where $\Pi$ is the tensor algebra over $\Lambda$ of a certain $\Lambda$-bimodule. This applies in particular to $d$-representation infinite algebras introduced in \cite{HIO12} and gives a triangle equivalence linking a $d$-representation infinite algebra $\Lambda$ with its associated $(d+1)$-preprojective algebra $\Pi$. In particular, when $\Pi$ is Noetherian we obtain an equivalence 
\begin{equation}\label{Minamotointro} \xymatrix{\Db(\mod \Lambda)\ar[r]^-{\sim} & \Db(\qgr \Pi)}.\end{equation}
 
On the other hand, in the paper \cite{AIR11}, for a $d$-representation infinite algebra $\Lambda$ with Noetherian preprojective algebra $\Pi$, we construct a triangle equivalence 
 \begin{equation}\label{AIRintro} \xymatrix{\Db(\mod \Lambda/\Lambda e\Lambda) \ar[r]^-{\sim} & \Sing^{\gr}(e\Pi e)},\end{equation}
where $e$ is an idempotent of $\Lambda$ satisfying some finiteness conditions. These conditions ensure in particular that the categories $\qgr \Pi$ and $\qgr e\Pi e$ are equivalent and that the restriction functor $$\xymatrix{\Db(\mod \Lambda/\Lambda e \Lambda)\ar[r]& \Db(\mod \Lambda)}$$ is fully faithful. Therefore combining the equivalences \eqref{Minamotointro} and \eqref{AIRintro} we obtain an embedding 
\begin{equation}\label{embeddingintro} \xymatrix{\Sing^{\gr}(e\Pi e)\ar@{^(->}[r] & \Db(\qgr e\Pi e)}.\end{equation}
The aim of this note is to show that this embedding is a particular case of Orlov's functor \eqref{Orlovintro}. This description in term of restriction functors gives a better insight of different properties of the functor \eqref{Orlovintro}, and permits, for instance, to understand the action of the degree shift functor of $\Sing^\gr(R)$ inside $\Db(\qgr R)$.

The plan of the paper is the following. We start in Section~\ref{section1} by recalling the definition of higher representation finite algebras and their preprojective algebras and state Minamoto's equivalence in this particular case. In Section~\ref{section2}, we recall results of \cite{AIR11} and deduce an embedding of type~\eqref{embeddingintro}. The main result of this note is given in Section~\ref{section3} where we prove that this embedding is the same as the one contructed by Orlov. This permits to recover some results of \cite{KMV08}. Some examples are treated in Section \ref{section4}. 
\subsection*{Acknowledgment}
I am grateful to Osamu Iyama for suggesting me the connection between the papers \cite{Orl05,Min11,AIR11}. I would also thank Helmut Lenzing and Idun Reiten for interesting conversations on the subject, and Lidia Angeleri-H\"ugel for explanations on recollements and for the reference \cite{AKL11}. 

\subsection*{Notation}
Throughout this paper $k$ is an algebraically closed field and all algebras are $k$-algebras. We denote by $D$ the $k$-dual, that is $D(-)=\Hom_k(-,k)$. 

Let $A$ be a $k$-algebra. All modules in this paper are right modules. We denote by $\mod A$ the category of finitely presented modules, by $\proj A$ the category of finitely generated projective $A$-modules and by $\fd A$ the category of finite dimensional $A$-modules. The envelopping algebra $A^{\rm op}\ten A$ is denoted $A^{\rm e}$. 

If $A$ is a graded $k$-algebra, we denote by $\gr A$ the category of finitely presented graded $A$-modules, and by $\gr\proj A$ the category of finitely generated projective graded $A$-modules. For a graded module $M=\bigoplus_{i\in\mathbb{Z}}\in\gr A$, we denote by $M(1)$ the graded module whose graded pieces are given by $(M(1))_i=M_{i+1}$. 

We denote by $\Dd(-)$ the derived category and by $\Db(-)$ the bounded derived category.

\section{$d$-Representation Infinite algebras and preprojective algebras}\label{section1}
\begin{definition}\cite{HIO12} Let $d$ be a non negative integer. A finite dimensional algebra $\Lambda$ is said to be \emph{$d$-representation infinite} if the following two conditions hold:
\begin{itemize}
\item $\gldim \Lambda \leq d$
\item $\{\SSS_d^{-i}\Lambda,\ i\in\mathbb{N}\}\subset \mod \Lambda$,
\end{itemize} 
where $\SSS_d$ is the autoequivalence $\SSS\circ[-d]=-\lten D\Lambda[-d]$ of $\Db(\mod \Lambda)$. 
\end{definition}

\begin{definition}
Let $\Lambda$ be a $d$-representation infinite algebra. Its associated \emph{preprojective algebra} (also called $(d+1)$-preprojective algebra) is defined to be the tensor algebra \[\Pi=\Pi_{d+1}(\Lambda):=\Talg_\Lambda \Ext^{d}_\Lambda(D\Lambda,\Lambda). \]
\end{definition}
The algebra $\Pi$ is naturally positively graded. We recall that $\Pi$ is called \emph{left graded coherent} if the category $\gr \Pi$ is closed under kernels (see also \cite{HIO12} for equivalent definitions). We denote by $\qgr\Pi$ the quotient category of $\gr\Pi$ by the torsion subcategory $\fd\gr\Pi$ of finite-dimensional graded $\Pi$-modules, by $q:\gr\Pi\to\qgr\Pi$ the natural projection and by $\mathbf{q}:\Db(\gr\Pi)\to\Db(\qgr\Pi)$ the associated left derived functor. If $M$ and $N$ are in $\gr\Pi$, by definition we have $$\Hom_{\qgr\Pi}(qM,qN)=\lim_{p\to\infty}\Hom_{\gr\Pi}(M_{\geq p},N),$$ where $M_{\geq p}$ is the graded module $\bigoplus_{i\geq p}M_i$.

\begin{theorem}\cite[Thm 3.12]{Min11}\label{thmMinamoto}
Let $\Lambda$ be a $d$-representation infinite algebra and $\Pi$ its associated preprojective algebra. Assume that $\Pi$ is left graded coherent, then the triangle functor $ \mathbf{q}(-\lten_\Lambda \Pi):\Db(\mod \Lambda)\longrightarrow \Db(\qgr \Pi)$ is a triangle equivalence. Moreover there is a commutative diagram 
\begin{equation}\label{diagcommMinamoto}\xymatrix{\Db(\mod \Lambda)\ar[d]_{\SSS_d^{-1}}\ar[rr]^{\mathbf{q}(-\lten_\Lambda \Pi)} &&\Db(\qgr \Pi)\ar[d]^{(1)}\\ 
\Db(\mod \Lambda)\ar[rr]^{\mathbf{q}(-\lten_\Lambda \Pi)} &&\Db(\qgr \Pi)}. \end{equation}
\end{theorem}

\begin{remark}\label{remarkpreprojCY} The preprojective algebras of $d$-representation infinite algebras satisfies very nice properties. In particular they are bimodule $(d+1)$-Calabi-Yau of Gorenstein parameter $1$, that is they have global dimension $(d+1)$ and there is an isomorphism \[\RHom_{\Pi^{\rm e}}(\Pi,\Pi^{\rm e})[d+1]\simeq \Pi(1) \quad \textrm{in }\Dd(\gr \Pi^{\rm e}).\]
In fact, the preprojective construction gives a bijection between $d$-representation infinite algebras and bimodule $(d+1)$-Calabi-Yau algebras of Gorenstein parameter $1$ (cf \cite[Thm 3.5]{AIR11} and \cite[Thm 4.35]{HIO12}). 
\end{remark}

\section{The singularity category of a Gorenstein algebra}\label{section2}

\begin{definition}\cite{Orl05} Let $R=\oplus_{p\geq 0}R_p$ be a positively graded algebra. The singularity category  is defined to be the Verdier localisation \[ \Sing^{\gr}(R):=\Db(\gr R)/\Db(\gr\proj R).\]
We denote by $\pi:\Db(\gr R)\to \Sing^{\gr}(R)$ the localisation functor.
\end{definition}
 When $R$ is Gorenstein, that is when the injective dimension of $R$ is finite as right and left $R$-module, then the singularity category can be interpreted as the stable category of graded maximal Cohen-Macaulay $R$-modules \cite{Buc87}. 

\begin{theorem}\label{thmAIR}\cite{AIR11}
Let $\Lambda$ be a $d$-representation infinite algebra such that its preprojective algebra $\Pi$ is Noetherian. Assume there exists an idempotent $e$ in $\Lambda$ such that 
\begin{itemize}
\item[(a)] $\dim_k \underline{\Pi}<\infty$ where $\underline{\Pi}:=\Pi/\Pi e\Pi$;
\item[(b)] $e\Lambda (1-e)=0$.
\end{itemize}
Then the functor given by the composition \[\xymatrix{\Db(\mod \underline{\Lambda})\ar[r]^-{\rm Res.} & \Db(\mod \Lambda)\ar[rr]^{-\lten_\Lambda \Pi e} && \Db(\gr e\Pi e)\ar[r]^{\pi} & \Sing^{\gr}(e\Pi e) }\] is a triangle equivalence, where $\underline{\Lambda}:=\Lambda/\Lambda e\Lambda$ and the functor ${\rm Res}:\Db(\mod \underline{\Lambda})\to\Db(\mod \Lambda)$ is the restriction functor. Moreover there is a commutative diagram
\begin{equation}\label{diagcommAIR}\xymatrix{\Db(\mod \underline{\Lambda}) \ar[rr]^-{\pi(-\lten_\Lambda\Pi e)}\ar[d]_{\underline{\SSS}_d^{-1}} && \Sing^{\gr}(e\Pi e)\ar[d]^{(1)}\\
\Db(\mod \underline{\Lambda}) \ar[rr]^-{\pi(-\lten_\Lambda \Pi e)} && \Sing^{\gr}(e\Pi e),}\end{equation} where $\underline{\SSS}_d$ is the autoequivalence $-\lten_{\underline{\Lambda}}D\underline{\Lambda}[-d]$ of $\Db(\mod \underline{\Lambda})$.

\end{theorem}

Combining Theorems~\ref{thmMinamoto} and~\ref{thmAIR} we get the following consequence.

\begin{corollary}\label{correcollement} Under the hypotheses of Theorem~\ref{thmAIR} there is an embedding \[\xymatrix{\Sing^{\gr}(e\Pi e)\ar@{^(->}[r] & \Db(\qgr e\Pi e)}.\] If moreover the global dimension of the algebra $e\Lambda e$ is finite, there is a recollement of triangulated categories 
\[\xymatrix{\Sing^{\gr}(e\Pi e)\ar[r] & \Db(\qgr e\Pi e)\ar@<1ex>[l]\ar@<-1ex>[l]\ar[r] & \Db(\mod e\Lambda e)\ar@<1ex>[l]\ar@<-1ex>[l]}.\]
\end{corollary}

\begin{proof}
First of all, notice that any graded Noetherian ring is left graded coherent, therefore Theorem~\ref{thmMinamoto} applies in the setup of Theorem~\ref{thmAIR}. Then we show that the functor \begin{align*}\gr\Pi &\to \gr e\Pi e\\ M&\mapsto Me\end{align*} induces an equivalence $\qgr\Pi\simeq \qgr e\Pi e$. Indeed, the functor $$-\ten_{e\Pi e} e\Pi:\gr e\Pi e\longrightarrow \gr\Pi$$ induces an inverse since the natural map $M\ten_{\Pi}\Pi e \Pi\to M$ is an isomorphism in $\qgr\Pi$ by hypothesis (a) of Theorem~\ref{thmAIR}.

By assumption (b), the restriction functor ${\rm Res}:\Db(\mod \underline{\Lambda})\to\Db(\mod \Lambda)$  is fully faithful, hence we get an embedding \[\xymatrix{\Sing^{\gr}(e\Pi e)\ar@{^(->}[r] & \Db(\qgr e\Pi e)}.\]

 Finally, we have $e\Lambda\simeq e\Lambda e\oplus e\Lambda(1-e)\simeq e\Lambda e$ as left $e\Lambda e$-module, so we have $\Lambda e\lten_{e\Lambda e}e\Lambda\simeq \Lambda e\ten_{e\Lambda e}e\Lambda\simeq \Lambda e \Lambda$. Therefore by Lemma 3.4 of \cite{AKL11}, if the global dimension of $e\Lambda e$ is finite there is a recollement
\[\xymatrix{ \Db(\mod \Lambda/\Lambda e \Lambda) \ar[r] & \Db(\mod \Lambda)\ar@<1ex>[l]\ar@<-1ex>[l]\ar[r] & \Db(\mod e\Lambda e)\ar@<1ex>[l]\ar@<-1ex>[l]},\]
and thus a recollement \[\xymatrix{\Sing^{\gr}(e\Pi e)\ar[r] & \Db(\qgr e\Pi e)\ar@<1ex>[l]\ar@<-1ex>[l]\ar[r] & \Db(\mod e\Lambda e)\ar@<1ex>[l]\ar@<-1ex>[l]}.\]
\end{proof}

\section{Orlov's orthogonal decomposition}\label{section3}

In this section, we recall Orlov's construction of embeddings \eqref{Orlovintro}, for graded Gorenstein algebras $R=k\oplus R_1\oplus R_2\oplus\ldots$ with positive Gorenstein parameter. Then we prove the main result of this note which links this construction with the results of \cite{Min11} and \cite{AIR11}. 

\begin{definition} A graded Noetherian algebra $R=k\oplus R_1\oplus R_2\oplus\ldots$ is said to be \emph{Gorenstein of parameter $1$} if it has injective dimension $(d+1)$ as left and right module and if there is an isomorphism
\[ \RHom_R(k,R)[d+1]\simeq k(1) \quad \textrm{in }\Db(\mod R^{\rm op}).\]
\end{definition}

\begin{theorem}\cite[Thm 2.5]{Orl05}\label{thmOrlov}
Let $R=k\oplus R_1\oplus R_2\oplus\ldots$ be a positively graded algebra, Noetherian and Gorenstein of parameter $1$. Then for any $i\in\mathbb{Z}$ there exists a fully faithful functor
\[\xymatrix{\Phi_i:\Sing^{\gr}(R)\ar[r]& \Db(\qgr R)}\] and a semi-orthogonal decomposition \[\Db(\qgr R)=\langle \mathbf{q} R(-i), \Phi_i(\Sing^{\gr}(R))\rangle.\]

\end{theorem}

The aim of this section is to prove that the recollement produced in Corollary~\ref{correcollement} is actually induced by one of the semi-orthogonal decomposition above.

\begin{theorem}\label{mainresult}
Let $\Lambda$ be a $d$-representation infinite algebra such that its preprojective algebra $\Pi$ is Noetherian. Assume that there exists an idempotent $e$ such that 
\begin{itemize}
\item[(a)] $\dim_k \underline{\Pi}<\infty$, where $\underline{\Pi}:=\Pi/\Pi e \Pi$;
\item[(b)] $e\Lambda (1-e)=0$;
\item[(c)] $e\Lambda e\simeq k$.
\end{itemize}
Then the graded algebra $R:=e\Pi e$ satisfies the hypothesis of Theorem~\ref{thmOrlov} and we have a commutative diagram
\begin{equation}\label{diagcomm}\xymatrix{\Db(\mod\underline{\Lambda})\ar[r]^-{\rm Res}\ar[d]_{\pi(-\lten_\Lambda\Pi e)} & \Db(\mod \Lambda)\ar[d]^{\mathbf{q}(-\lten_\Lambda\Pi e)}\\ \Sing^{\gr}(R)\ar[r]^{\Phi_0} & \Db(\qgr R)}.\end{equation}
\end{theorem}

\begin{proof}
The algebra $R$ is Noetherian since for any right (resp. left) ideal $I$ of $R$, $I\Pi$ (resp. $\Pi I$) is a right (resp. left) ideal of $\Pi$. Moreover if $I_1\subset \ldots \subset I_\ell$ is a striclty ascending chain of ideals in $R$, then so are $I_1\Pi\subset\ldots\subset I_\ell\Pi$ and $\Pi I_1\subset\ldots\subset \Pi I_\ell$. Hence the noetherianity of $\Pi$ implies the noetherianity of $R$.
  
First of all we have $R_0=e\Pi_0 e=e\Lambda e=k$.

Now we prove that $R$ is Gorenstein of parameter $1$. By Remark~\ref{remarkpreprojCY}, the algebra $\Pi$ is bimodule $(d+1)$-Calabi-Yau of Gorenstein parameter $1$. Then by \cite[Remark 2.7]{AIR11}, we deduce an isomorphism 
\[\RHom_{R^{\rm e}}(R,R^{\rm e})[d+1]\simeq R(1) \quad \textrm{in }\Dd(\gr R^{\rm e}).\]
Applying the functor $-\lten_{R}R_0=\lten_R k$ we obtain the following isomorphisms in $\Dd(\gr R^{\rm op}\ten R_0)=\Dd(\gr R^{\rm op})$:
\begin{align*}k(1)[-d-1] &\simeq \RHom_{R^{\rm e}}(R,R^{\rm e})\lten_R R_0\\
 & \simeq \RHom_{R_0^{\rm op}\ten R}(R_0,R_0^{\rm op}\ten R)\\ & \simeq \RHom_R(k,R).
\end{align*}
Therefore we are in the setup of Theorem~\ref{thmOrlov}. In order to check that the diagram \eqref{diagcomm} is commutative, we have to recall the construction of the functors $\Phi_i$. We use the same notation as in \cite{Orl05} that we recall here for the convenience of the reader. 

For $i\in \mathbb{Z}$ we denote by $\gr_{\geq i} R$ the full subcategory of $\gr R$ which consists of all modules $M\in\gr R$ such that $M_p=0$ when $p<i$. We denote by $\Pp$ the category $\Db(\gr\proj R)$, and by $\Pp_{\geq i}$ the full subcategory of $\Pp$ generated by the free modules $R(p)$ with $p\geq -i$. 

Orlov proves the existence of the following semi-orthogonal decompositions (Lemmas 2.3 and 2.4 in \cite{Orl05})
\[\Db(\gr R)=\langle \Db(\gr_{\geq i} R),\Pp_{<i}\rangle, \quad \Db(\gr_{\geq i} R)=\langle \Pp_{\geq i}, ^\perp\Pp_{\geq i}\rangle\] and proves the equivalence $\Dd(\gr R)/\Pp\simeq \Db(\gr_{\geq i} R)/\Pp_{\geq i}$.
 Then the funcotr $\Phi_i$ is given by the following composition
\[ \xymatrix{\Sing^{\gr}(R)=\Dd(\gr R)/\Pp\ar[r]^-\sim &\Db(\gr_{\geq i} R)/\Pp_{\geq i}\simeq ^\perp \Pp_{\geq i} \ar@{^(->}[r]& \Db(\gr R)\ar[r]^{\mathbf{q}} & \Db(\qgr R)}.\]

More precisely, if $X\in\Db(\gr R)$ then there exists a triangle $$\xymatrix{Y\ar[r] & X\ar[r]& X_1\ar[r] &Y[1]}$$ in $\Db(\gr R)$ with $Y\in \Pp_{<i}$ and $X_1\in \Db(\gr_{\geq i} R)$. Since $\Pp_{< i}\subset \Pp$, $X$ and $X_1$ are isomorphic in the singularity category. Then there exists a triangle $$\xymatrix{X_2\ar[r] & X_1\ar[r] & Z\ar[r]& X_2[1]}$$ where $Z\in\Pp_{\geq i}$ and $X_2\in ^\perp\Pp_{\geq i}$. The objects $X$ and $X_2$ are isomorphic in the singularity category. Then $\Phi_i(\pi X)$ is defined to be $\mathbf{q}(X_2)$. In particular if $X\in^\perp\Pp_{\geq i}$ then $\Phi_i\circ \pi(X)\simeq \mathbf{q}(X)$.

Now let $X\in \Db(\mod \Lambda/\Lambda e \Lambda)$. The object $X\lten_\Lambda \Pi e$ belongs to the thick subcategory of $\Db(\gr R)$ generated by $(1-e)\Pi e$. For $\ell<0$, the space $\Hom_{\gr R}((1-e)\Pi e, R(\ell))$ clearly vanishes since $\Pi$ hence $R$ are positively graded. Moreover we have 
\begin{align*}\Hom_{\gr R}((1-e)\Pi e, R) & = \Hom_{\gr R}((1-e)\Pi e, e\Pi e)  \\ & \simeq (e\Pi (1-e))_0 \\ & =e\Lambda (1-e)=0 \quad\textrm{by assumption (b).}\end{align*}   
This means that $(1-e)\Pi e$ belongs to the category \[^\perp\Pp_{\geq 0}=\{M\in \Db(\gr R)\mid \Hom_{\Db(\gr R)}(M, R(\ell)), \forall \ell\leq 0  \}\] and so does $X\lten_\Lambda \Pi e$. Therefore $\Phi_0\circ \pi (X\lten_\Lambda \Pi e)=\mathbf{q}(X\lten_{\Lambda}\Pi e)$ and the diagram \eqref{diagcomm} is commutative.

\end{proof}

The functor $\Phi_0$ does not commute with the degree shift. Using diagrams \eqref{diagcommMinamoto} and \eqref{diagcommAIR}, we can deduce what is the degree shift action of the category $\Sing^{\gr}(R)$ inside the category $\Db(\qgr R)$. More precisely we recover the following result.

\begin{corollary}\cite[Lemma 5.2.1]{KMV08}
In the setup of Theorem~\ref{mainresult}, let $M$ be an object in $\Sing^{\gr}(R)$. Then we have 
\[\Phi_0(M(1))\simeq {\sf Cone}(\RHom_{\qgr R}(\mathbf{q}R,\Phi_0(M)(1))\to \Phi_0(M)(1)).\]
\end{corollary}

\begin{proof}
Using the diagrams \eqref{diagcomm}, \eqref{diagcommMinamoto} and \eqref{diagcommAIR}, it is enough to understand the action of the functor $\underline{\SSS}_d$ inside the category $\Db(\mod \Lambda)$. 
We denote by $\theta$ a projective resolution of the object $\RHom_{\Lambda}(D\Lambda,\Lambda)$ in $\Db(\mod \Lambda^{\rm e})$, and by $\underline{\theta}$ a projective resolution of the object $\RHom_{\underline{\Lambda}}(D\underline{\Lambda},\underline{\Lambda})$ in $\Db(\mod \underline{\Lambda}^{\rm e})$. Then for $N\in\Db(\mod\underline{\Lambda})$ we have the following isomorphisms in $\Db(\mod \Lambda)$
\begin{align*}
\underline{\SSS}^{-1}_dN & \simeq N\ten_{\underline{\Lambda}}\underline{\theta}[d] \\
 & \simeq N\ten_{\underline{\Lambda}}\underline{\Lambda}\ten_{\Lambda}\theta\ten_{\Lambda}\underline{\Lambda}[d] \quad \textrm{by \cite[Lemma 4.2]{AIR11}}\\ &\simeq N\ten_{\Lambda}\theta\ten_{\Lambda}\underline{\Lambda}[d]\\
&\simeq {\sf Cone}(N\ten_{\Lambda}\theta\ten_{\Lambda}\Lambda e\Lambda \to N\ten_{\Lambda}\theta)[d]\quad \textrm{induced by the inclusion }\Lambda e\Lambda\to\Lambda\\
 &\simeq {\sf Cone}(\SSS_d^{-1} N\ten_{\Lambda}\Lambda e \Lambda\to\SSS_d^{-1}N)\\
&\simeq {\sf Cone }(\RHom_{\Lambda}(e\Lambda,\SSS_d^{-1}N)\ten_{k}e\Lambda\to \SSS_d^{-1}N) 
\end{align*}
Now, we can translate this isomorphism using the commutative diagrams \eqref{diagcomm}, \eqref{diagcommMinamoto} and \eqref{diagcommAIR}. And we obtain for any $M\in\Sing^{\gr}(R)$
\begin{align*}
\Phi_0(M(1)) &\simeq {\sf Cone}(\RHom_{\qgr R}(\mathbf{q}(e\Lambda\lten_{\Lambda}\Pi e),\Phi_0(M)(1))\ten_k \mathbf{q}(e\Lambda\lten_{\Lambda}\Pi e)\to \Phi_0(M)(1))\\ & \simeq {\sf Cone}(\RHom_{\qgr R}(\mathbf{q}R,\Phi_0(M)(1))\ten_k \mathbf{q}R\to \Phi_0(M)(1)).
\end{align*}
\end{proof}

Using the previous description, it is also possible to describe for any $i\in\mathbb{Z}$ the functors $\Phi_i$ defined by Orlov in terms of the categories $\Db(\mod \Lambda)$ and $\Db(\mod \underline{\Lambda})$.
\begin{corollary}
Under hypothesis and notations of Theorem~\ref{mainresult}, for any $i\in\mathbb{Z}$ there is a commutative diagram
\[\xymatrix{\Db(\mod\underline{\Lambda})\ar[rr]^-{\SSS_d^{-i}\circ{\rm Res}\circ \underline{\SSS}_d^{i}}\ar[d]_{F}& & \Db(\mod \Lambda)\ar[d]^{G}\\ \Sing^{\gr}(R)\ar[rr]^{\Phi_i} && \Db(\qgr R)}\]
\end{corollary}

\begin{proof}
Using the commutativity of the diagrams \eqref{diagcomm}, \eqref{diagcommMinamoto} and \eqref{diagcommAIR}, it is enough to check that for any $M\in \Sing^{\gr}(R)$ we have $\Phi_i(M)\simeq \Phi_0(M(-i))(i)$. Indeed one immediately checks that if $N\in ^\perp\Pp_{\geq i}$, then $N(-i)\in ^\perp\Pp_{\geq 0}$, so $$\Phi_0(\pi N (-i))(i)=\mathbf{q}(N(-i))(i)=\mathbf{q}(N)=\Phi_i(\pi N).$$
\end{proof}

\section{Examples}\label{section4}

In this section, $d\geq 1$ is an integer and we denote by $S$ the polynomial algebra $S=k[X_0,\ldots,X_d]$.

Let $\Lambda$ be the $d$-Beilinson algebra, that is the algebra presented by the quiver 
\[\scalebox{0.5}{
\begin{tikzpicture}[>=stealth, scale=1]
\node (P1) at (0,0){$0$}; \node (P2) at (5,0){$1$}; \node (P3) at
(10,0){$2$};

\node(P6) at (20,0){$d-1$}; \node(P7) at (25,0){$d$};  

\draw [->] (0.3,0.3)--node[fill=white,inner sep=.5mm]{$x_{0,0}$}(4.7,0.3);
\draw [->] (0.3,0)--node[fill=white,inner sep=.5mm,]{$x_{1,0}$}(4.7,0);
\draw [->] (0.3,-0.6)--node[fill=white,inner sep=.5mm]{$x_{d,0}$}(4.7,-0.6);

\draw [->] (5.3,0.3)--node[fill=white,inner sep=.5mm]{$x_{0,1}$}(9.7,0.3);
\draw [->] (5.3,0)--node[fill=white,inner sep=.5mm,]{$x_{1,1}$}(9.7,0);
\draw [->] (5.3,-0.6)--node[fill=white,inner sep=.5mm]{$x_{d,1}$}(9.7,-0.6); \draw[loosely dotted](11,0)--(19,0);
\draw[loosely dotted] (2,0)--(2,-0.6); \draw[loosely dotted]
(3,0)--(3,-0.6); \draw[loosely dotted] (7,0)--(7,-0.6);
\draw[loosely dotted] (8,0)--(8,-0.6); \draw[loosely dotted]
(22,0)--(22,-0.6); \draw[loosely dotted] (23,0)--(23,-0.6); \draw
[->] (21,0.3)--node[fill=white,inner sep=.5mm]{$x_{0,d-1}$}(24,0.3);
\draw [->] (21,0)--node[fill=white,inner sep=.5mm,]{$x_{1,d-1}$}(24,0);
\draw [->] (21,-0.6)--node[fill=white,inner sep=.5mm]{$x_{d,d-1}$}(24,-0.6);

\end{tikzpicture}}
\] with relations $x_{i,\ell+1} x_{j,\ell}-x_{j,\ell+1}x_{i,\ell}=0$ for $\ell =0,\ldots, d-2$ and $i,j=0,\ldots, d$. This algebra is a $d$-representation infinite algebra (cf Example 2.15 of~\cite{HIO12}).

Its associated preprojective algebra is the graded algebra $\Pi$ given by the graded quiver

\[\scalebox{.5}{
\begin{tikzpicture}[>=stealth, scale=1]
\node (P1) at (0,0){$1$}; \node (P2) at (5,0){$1$}; \node (P3) at
(10,0){$3$};

\node(P6) at (20,0){$d-2$}; \node(P7) at (25,0){$d-1$}; \node(P0) at
(12,4){$0$}; 

\draw [->] (0.3,0.3)--node[fill=white,inner
sep=.5mm,]{$x_{0,1}$}(4.7,0.3);
\draw [->] (0.3,0)--node[fill=white,inner sep=.5mm,]{$x_{1,1}$}(4.7,0);
\draw [->] (0.3,-0.6)--node[fill=white,inner
sep=.5mm,]{$x_{d,1}$}(4.7,-0.6);

\draw [->] (5.3,0.3)--node[fill=white,inner
sep=.5mm,]{$x_{0,2}$}(9.7,0.3);
\draw [->] (5.3,0)--node[fill=white,inner sep=.5mm,]{$x_{1,2}$}(9.7,0);
\draw [->] (5.3,-0.6)--node[fill=white,inner
sep=.5mm,]{$x_{d,2}$}(9.7,-0.6); \draw[loosely dotted](11,0)--(19,0);
\draw[loosely dotted] (2,0)--(2,-0.6); \draw[loosely dotted]
(3,0)--(3,-0.6); \draw[loosely dotted] (7,0)--(7,-0.6);
\draw[loosely dotted] (8,0)--(8,-0.6); \draw[loosely dotted]
(22,0)--(22,-0.6); \draw[loosely dotted] (23,0)--(23,-0.6); 
\draw
[->] (21,0.3)--node[fill=white,inner sep=.5mm,]{$x_{0,d-1}$}(24,0.3);
\draw [->] (21,0)--node[fill=white,inner sep=.5mm,]{$x_{1,d-1}$}(24,0);
\draw [->] (21,-0.6)--node[fill=white,inner
sep=.5mm,]{$x_{d,d-1}$}(24,-0.6);

\draw[->] (11,4)--node[fill=white,inner
sep=.5mm,xshift=-4mm]{$x_{0,0}$}(0,1); \draw[->]
(11.2,3.8)--node[fill=white,inner sep=.5mm,]{$x_{1,0}$}(0.2,0.8);
\draw[->] (11.6,3.4)--node[fill=white,inner
sep=.5mm,]{$x_{d,0}$}(0.6,0.4);

\draw[<-] (13,4)--node[fill=white,inner
sep=.5mm,xshift=4mm]{$x_{0,d}$}(25,1); \draw[<-]
(12.8,3.8)--node[fill=white,inner sep=.5mm,]{$x_{1,d}$}(24.8,0.8);
\draw[<-] (12.4,3.4)--node[fill=white,inner
sep=.5mm,]{$x_{d,d}$}(24.4,0.4);

\end{tikzpicture}}
\]
where ${\rm deg}(x_{i,d})=1$ for any $i$ and ${\rm deg}(x_{i,\ell})=0$ for any $i$ and $\ell\leq d-1$ and with relations $x_{i,\ell+1} x_{j,\ell}-x_{j,\ell+1}x_{i,\ell}=0$ for $i,j,\ell\in\mathbb{Z}/(d+1)\mathbb{Z}$.

Let $e$ be the idempotent associated with the vertex $0$. The algebra $\underline{\Pi}=\Pi/\Pi e\Pi$ is presented by an acyclic quiver, so it is finite-dimensional. We also clearly have $e\Lambda (1-e)=0$ and $e\Lambda e=k$. So we are in the setup of Theorem~\ref{mainresult}.

One easily checks that the algebra $e\Pi e$ is isomorphic to the subalgebra of $S$ generated by the monomials of the form $\prod_{i=0}^d X_i^{\alpha_i}$ where $\sum_{i=0}\alpha_i=d+1$. Moreover, using the fact that ${\rm deg}(x_{i_0,0}x_{i_1,1}\ldots x_{i_d,d})=1$ in $\Pi$ for any $i_0,\ldots, i_d$, we deduce that $e\Pi e$ is isomorphic as graded ring to the $(d+1)$-Veronese algebra $S^{(d+1)}$.

Therefore one has a triangle equivalence \begin{equation}\label{minamotoeq}\Db(\mod\Lambda)\simeq \Db(\qgr S^{(d+1)}).\end{equation} By Serre's theorem, \cite{Serre} one has $\qgr R\simeq {\sf coh}({\rm Proj}R)$ if $R$ is generated in degree $1$ and $R_0\simeq k$. So we have equivalences $\qgr S^{(d+1)}\simeq {\sf coh}({\rm Proj} S^{(d+1)})\simeq {\sf coh}(\rm Proj S)\simeq {\sf coh}\mathbb{P}^d$. Therefore the equivalence \eqref{minamotoeq} can be written \[\Db(\mod \Lambda)\simeq \Db({\sf coh}\mathbb{P}^d)\] and we recover the triangle equivalence due to Beilinson \cite{Bei}.

The algebra $\underline{\Lambda}=\Lambda/\Lambda e\Lambda$ is presented by the quiver \[\scalebox{0.5}{
\begin{tikzpicture}[>=stealth, scale=1]
 \node (P2) at (5,0){$1$}; \node (P3) at
(10,0){$2$};

\node(P6) at (20,0){$d-1$}; \node(P7) at (25,0){$d$};

\draw [->] (5.3,0.3)--node[fill=white,inner
sep=.5mm,]{$x_{0,1}$}(9.7,0.3);
\draw [->] (5.3,0)--node[fill=white,inner sep=.5mm,]{$x_{1,1}$}(9.7,0);
\draw [->] (5.3,-0.6)--node[fill=white,inner
sep=.5mm,]{$x_{d,1}$}(9.7,-0.6); \draw[loosely dotted](11,0)--(19,0);
\draw[loosely dotted] (7,0)--(7,-0.6);
\draw[loosely dotted] (8,0)--(8,-0.6); \draw[loosely dotted]
(22,0)--(22,-0.6); \draw[loosely dotted] (23,0)--(23,-0.6); \draw
[->] (21,0.3)--node[fill=white,inner sep=.5mm,]{$x_{0,d-1}$}(24,0.3);
\draw [->] (21,0)--node[fill=white,inner sep=.5mm,]{$x_{1,d-1}$}(24,0);
\draw [->] (21,-0.6)--node[fill=white,inner
sep=.5mm,]{$x_{d,d-1}$}(24,-0.6);

\end{tikzpicture}}
\] with relations $x_{i,\ell+1} x_{j,\ell}-x_{j,\ell+1}x_{i,\ell}=0$. By Theorem~\ref{thmAIR} there is a triangle equivalence \[\Sing^{\gr }(S^{(d+1)})\simeq \Db(\mod\underline{\Lambda})\] and we deduce a recollement  \[\xymatrix{\Sing^{\gr}(S^{(d+1)})\ar[r] & \Db({\sf coh}\mathbb{P}^d)\ar@<1ex>[l]\ar@<-1ex>[l]\ar[r] & \Db(\mod k)\ar@<1ex>[l]\ar@<-1ex>[l]}.\]

Note that for $d=1$, one easily checks that the graded algebra $S^{(2)}$ is isomorphic to the graded algebra $k[U,V,W]/(UW-V^2)$ where ${\rm deg}(U)={\rm deg}(V)={\rm deg}(W)=1$. Thus $Y={\rm Proj}(S^{(2)})\simeq \mathbb{P}^1$ can be seen  as a Fano hypersurface of $\mathbb{P}^2$ of degree $2$. Then we are in the setup of Theorem 3.11 (i) of \cite{Orl05}, and the graded singularity category of $S^{(2)}$ is equivalent to the triangulated category of graded B-branes.

\end{document}